\documentclass[12pt,a4paper,leqno,verbatim]{amsart}
\newcounter{minutes}
\divide\time by 60
\newcounter{hours}
\multiply\time by 60 \addtocounter{minutes}{-\time}

\usepackage{amssymb}
\usepackage{hyperref}
\usepackage{graphicx}
\usepackage{color}

\date{}
\newfont{\cyrilic}{wncyr10 scaled 1000}
\title[On some inequalities for the identric and logarithmic means]
{On some inequalities for the identric, logarithmic and related means}

\author[J. S\'andor]{J\'ozsef S\'andor}
\address{Babe\c{s}-Bolyai University
Department of Mathematics
Str. Kogalniceanu nr. 1
400084 Cluj-Napoca, Romania}
\email{jsandor@math.ubbcluj.ro}

\author[B.A Bhayo]{Barkat Ali Bhayo}
\address{Department of Mathematical Information Technology, University of Jyv\"askyl\"a, 40014 Jyv\"askyl\"a, Finland\\}
\email{bhayo.barkat@gmail.com}

\newcommand{\comment}[1]{}

\swapnumbers
\theoremstyle{plain}

\newtheorem{theorem}[equation]{Theorem}
\newtheorem{theorem a}[equation]{Theorem A}
\newtheorem{theorem b}[equation]{Theorem B}
\newtheorem{lemma}[equation]{Lemma}

\newtheorem{corollary}[equation]{Corollary}

\newtheorem{remark}[equation]{Remark}

\numberwithin{equation}{section}

\begin{document}


%
%
\def\thefootnote{}
\footnotetext{ \texttt{\tiny File:~\jobname .tex,
          printed: \number\year-\number\month-\number\day,
          \thehours.\ifnum\theminutes<10{0}\fi\theminutes}
} \makeatletter\def\thefootnote{\@arabic\c@footnote}\makeatother


\begin{abstract}
We offer new proofs, refinements as well as new results related to classical means of two variables, including the identric and logarithmic means.
\end{abstract}

\maketitle

\bigskip
{\bf 2010 Mathematics Subject Classification}: 26D05, 26D15, 26D99.

{\bf Keywords}: Means of two arguments, inequalities for means, identities for means, trigonometric and hyperbolic inequalities.


\section{\bf Introduction}
Since last few decades, the inequalities involving the classical means
such as arithmetic mean $A$, geometric mean $G$, identric mean $I$ and logarithmic mean $L$ and weighted geometric mean $S$ have been  studied extensively by numerous authors, e.g. see \cite{alzer1, alzer2,carlson,ns1004, ns1004a,sandorc,sandord,sandore}. 

For two positive real numbers $a$ and $b$, we define
$$A=A(a,b)=\frac{a+b}{2},\quad G=G(a,b)=\sqrt{ab},$$
$$L=L(a,b)=\frac{a-b}{\log(a)-\log(y)},\quad a\neq b,$$
$$I=I(a,b)=\frac{1}{e}\left(\frac{a^a}{b^b}\right)^{1/(a-b)},\quad a\neq b,$$
$$S=S(a,b)=  (a^a b^b)^{1/(a+b)}.$$
For the historical background of these means we refer the reader to 
\cite{alzer2,carlson,mit,sandor611,sandorc,sandord,sandore}. Generalizations, or related means are studied in \cite{newmean,ns1004a,ns1004,ns1605,sandor611,sandornew,sandor1405}. Connections of these means with trigonometric or hyperbolic inequalities are pointed out in \cite{newmean,sandora,ns0206,sandornew,sandore}.

Our main result reads as follows:
\begin{theorem}\label{thm1}
For all distinct positive real numbers $a$ and $b$, we have
\begin{equation}\label{thmineq2905}
1< \frac{I}{\sqrt{I(A^2,G^2)}} < \frac {2}{\sqrt{e}}.
\end{equation}
Both bounds are sharp.
\end{theorem}

\begin{theorem}\label{thm2} For all distinct positive real numbers $a$ and $b$, we have
\begin{equation}\label{thm2ineq}
1<\frac{2I^2}{A^2+G^2}<c
\end{equation}
where $c=1.14\ldots$.
The bounds are best possible.
\end{theorem}

\begin{remark} \rm
{\bf A.} The left side of \eqref{thm2ineq} may be rewritten also as 
\begin{equation} \label{ineq3005a}
I>	Q(A,G),  
\end{equation}
where $Q(x,y)=\sqrt{(x^2+y^2)/2}$ denotes the root square mean of $x$ and $y$.
In 1995, Seiffert \cite{seiff0} proved  
the first inequality in \eqref{thmineq2905}
by using series representations, which is rather strong.
Now we prove that, \eqref{ineq3005a} is a refinement of the first inequality in \eqref{thmineq2905}. 
Indeed, by the known relation $I(x,y)< A(x,y)= (x+y)/2$, we can write 
$$I(A^2,G^2)< (A^2+G^2)/2 = Q(A,G)^2,$$
so one has:
\begin{equation} \label{ineq3005c}
I>Q(A,G) > \sqrt{I(A^2,G^2)}.
\end{equation}
As  we have $I(x^2,y^2)> I(x,y)^2$  (see S\'andor \cite{sandorc}), hence \eqref{ineq3005c} offers also a refinement of 
\begin{equation} \label{ineq3005d}
I> I(A,G). 
\end{equation}
Other refinements of \eqref{ineq3005d} have been provided in a paper by Neuman and S\'andor 
\cite{ns1605}. Similar inequalities involving the logarithmic mean, as well as S\'andor's means X  and Y , we quote \cite{newmean,sandora,sandornew}. In the second part of paper, similar results will be proved.

\noindent{\bf B.}  In 1991, S\'andor \cite{sandord} proved the inequality
\begin{equation} \label{ineq3005e}
I> (2A+G)/3.   
\end{equation}
It is easy to see that, the left side of \eqref{thm2ineq} and \eqref{ineq3005e} cannot be compared.

In 2001 S\'andor and Trif  \cite{st} have proved the following inequality:
\begin{equation} \label{ineq3005f}
I^2< (2A^2+G^2)/3 .
\end{equation}
The left side of \eqref{thm2ineq} offers a good companion to \eqref{ineq3005f}.
We note that the inequality \eqref{ineq3005f} and the right side of \eqref{thm2ineq} cannot be compared.
\end{remark}

In \cite{seiff0}, Seiffert proved the following relation:
\begin{equation}\label{611ineqa}
L(A^2,G^2)>  L^2,
\end{equation}
which was refined by Neuman and S\'andor \cite{ns1605} 
(for another proof, see\cite{ns1004a}) as follows:
\begin{equation}\label{611ineqb}
L(A,G)> L.
\end{equation}
We will prove with a new method the following refinement of \eqref{611ineqa} and a counterpart of \eqref{611ineqb}:

\begin{theorem}\label{thm611} We have
\begin{equation}\label{in0611}
L(A^2,G^2) =\frac{(A+G)}{2}L(A,G)>  \frac{(A+G)}{2}L >  L^2,
\end{equation}
\begin{equation}\label{in0611a}
L(I,G)  <  L,
\end{equation}
\begin{equation}\label{in-11}
L< L(I,L)<  L\cdot (I-L)/(L-G).
\end{equation}
\end{theorem}

\begin{corollary}\label{coro} One has
\begin{equation}\label{ineq711a}
G\cdot I/L <  \sqrt{I\cdot G} <L(I,G)  <  L,
\end{equation}
\begin{equation}\label{ineq711b}
(L(I,G))^2<  L\cdot L(I,G)< L(I^2,G^2)<  L\cdot (I+G)/2.
\end{equation}
\end{corollary}

\begin{remark} \rm
\noindent {\bf A.} Relation \eqref{ineq711a} improves the inequality  
$$G\cdot I/L <  L(I,G),$$
due to Neuman and S\'andor \cite{ns1605}.
Other refinements of the inequality 
\begin{equation}\label{in-10}
L< (I+G)/2
\end{equation} 
are provided in \cite{sandorF}. 

\noindent {\bf B.} Relation \eqref{in-11} is indeed a refinement of \eqref{in-10}, as the weaker inequality can be written as $(I-L)/(L-G)>1$, which is in fact \eqref{in-10}. 
\end{remark}

The mean $S$ is strongly related to other classical means. For example, in 1993  S\'andor \cite{sandore} discovered the identity
\begin{equation}\label{1-in-1011}
S(a,b)=  I(a^2,b^2)/I(a,b),
\end{equation}
where $I$ is the identric mean. Inequalities for the mean $S$ may be found in 
\cite{sandorc,sandore,rasa}.

The following result shows that  $I$  and $S(A,G)$ cannot be compared , but this is not true in case of $I$ and $S(Q,G)$. Even a stronger result holds true.

\begin{theorem}\label{thm-1011}
None of the inequalities  
$I>S(A, G)$  or  $I< S(A,G )$ holds  true. 
On the other hand, one  has
\begin{equation}\label{2-in-1011}
S(Q,G)  >A >I,
\end{equation}
\begin{equation}\label{3-in-1511}
I(Q,G) <A.
\end{equation}
\end{theorem}

\begin{remark}\rm
By \eqref{2-in-1011} and \eqref{3-in-1511}, one could ask if $I$ and $I(Q,G)$ may be compared to each other.
It is not difficult to see that, this becomes equivalent to one of the inequalities
\begin{equation}\label{1711}
\frac{ y\log y}{y-1} < ({\rm or}>)\frac{x}{\tanh(x)},\quad x>0,
\end{equation}
where $y=\sqrt{\cosh(2x)}$. By using the Mathematica Software \cite{ru}, we can show that \eqref{1711} with $``<"$ is not true for 
$x=3/2$,
while \eqref{1711} with $``>"$ is not true for $x=2$.
\end{remark}


\section{\bf Lemmas and proofs of the main results}

The following lemma will be utilized in our proofs.
\begin{lemma}\label{BasicLemma} For $b>a>0$ there exists an $x>0$ such that 
\begin{equation}\label{ineq1205a}
\frac{A}{G} = \cosh(x),\, \frac{I}{G}= e^{x/\tanh (x) -1}.  
\end{equation}
\end{lemma}

\begin{proof}
For any $a>b>0$, one can find an $x>0$ such that  $a=e^x\cdot G$ and 
$b=e^{-x}\cdot G$. Indeed, it is immediate that such an $x$ is (by considering
$a/b= e^{2x}$), $x=(1/2)\log (a/b) >0$.
Now, as $A=  G\cdot(e^x+e^{-x})/2= G\cosh(x)$, we get  
$A/G= \cosh(x).$
Similarly, we  get
$$I= G\cdot (1/e) \exp(x(e^x+ e^{-x})/(e^x-e^{-x})),$$ which gives
$I/G=e^{x/\tanh (x)-1}$.
\end{proof}

\noindent{\bf Proof of Theorem \ref{thm1}.} For $x>0$, we have $I/G= e^{x/\tanh (x) -1}$ and $A/G= \cosh (x)$ by 
Lemma \ref{BasicLemma}.
Since $$\log( I(a,b))= \frac{a\log a-b\log b}{a-b} -1,$$ we get  
$$\log ( \sqrt {I((A/G)^2,1)}) = \frac{\cosh (x)^2 \log(\cosh (x))}{\cosh (x)^2 -1} -\frac{1}{2}.$$
 By using this identity, and taking the logarithms in the second identity 
of \eqref{ineq1205a}, the inequality   
$$0<  \log (I/G)  -\log (\sqrt{I(A/G)^2,1}) < \log 2-1/2 $$ 
becomes
\begin{equation}\label{ineq2905}
\frac{1}{2}< f(x)  < \log 2,
\end{equation}
where
$$f(x)=\frac{x}{\tanh(x)}-\frac{\log(\cosh(x))}{\tanh(x)^2}.$$
A simple computation (which we omit here) for the derivative of $f(x)$ gives:
\begin{equation}\label{ineq2905a}
\sinh (x)^3f'(x)= 2\cosh (x)\log(\cosh (x))- x\sinh(x). 
\end{equation}
The following inequality appears in \cite{ns0206}:
\begin{equation}\label{ineq2905b}
\log(\cosh (x))> \frac{x}{2}\tanh (x) ,\,\, x>0,
\end{equation}
which gives $f'(x)>0$, so $f(x)$ is strictly increasing in $(0, \infty)$.
As $\lim_{x\to 0} f(x)=1/2$, and $\lim_{x\to \infty} f(x)=\log 2$, the    
double inequality \eqref{ineq2905} follows.
So we have obtained a new proof of \eqref{thmineq2905}. $\hfill\square$

\bigskip

We note that Seiffert's proof is based on certain infinite series representations. 
Also, our proof shows that the constants $1$ and $2/\sqrt{e}$ in \eqref{thmineq2905} are optimal.

\begin{lemma}\label{lema2905} Let
$$f(x)=\frac{2x}{\tanh(x)}-\log\left(\frac{\cosh(x)^2+1}{2}\right),\,\,x>0.$$
Then 
\begin{equation}\label{lemineq2905}
2<f(x)<f(1.606\ldots)=2.1312\ldots.
\end{equation}
\end{lemma}

\begin{proof}
One has  $(\cosh(x)^2+1)/2f'(x)= g(x)$,
where  
$$g(x)= \sinh (x)\cosh (x)^3 –x\cosh (x)^2+ \sinh (x)\cosh (x)-x$$
$$-\cosh(x)\sinh (x)^3  
2\sinh (x)\cosh (x)-x\cosh (x)^2 –x,$$   
by remarking that  $$\sinh (x)\cosh (x)^3-\cosh (x)\sinh (x)^3 =\sinh (x)\cosh (x).$$ 

Now, a simple computation gives  
$$g'(x)= \sinh(x)\cdot(3\sinh(x)-2x\cosh (x))= 3\sinh (x)\cosh (x)\cdot k(x),$$  where 
$k(x)= \tanh (x)-2x/3$
As it is well known that the function $\tanh (x)/x$ is strictly decreasing, the equation   $\tanh (x)/x =2/3 $ can have at most a single solution.
As $\tanh (1)=0.7615\ldots> 2/3$ and $\tanh (3/2)= 0.9051\ldots<1= (2/3).(3/2)$, we find that the equation $k(x)=0$ has a single solution $x_0$ in $(1, 3/2)$, and also that $k(x)>0$
 for $x$ in $(0, x_0)$ and $k(x)<0$ in $(x_0, 3/2)$. This means that the function $g(x)$ is strictly increasing in the interval $(0, x_0)$ and strictly decreasing in $(x_0, \infty )$.  As $g(1)= 0.24…>0$, clearly $g(x_0)>0$, while $g(2)= -3.01..<0$ implies that there exists a single zero $x_1$ of $g(x)$ in $( x_0, 2)$. In fact, as $g(3/2)= 0.21…>0$, we get that $x_1$  is in $(3/2, 2)$.

From the above consideration we conclude that $g(x)>0$ for  $x\in( 0, x_1)$ and $g(x)<0$  for $x \in (x_1, \infty)$. Therefore, the point $x_1$ is a maximum point to the function $f(x)$.
It is immediate that $\lim_{x\to 0} f(x)=2.$
On the other hand, we shall compute the limit of f(x) at $\infty$.
Clearly $t=\cosh (x)$  tends to $\infty$ as $x$ tends to $\infty$. 
Since $\log (t^2 +1) - \log (t^2)= \log ((t^2+1)/t^2 )$ tends to $log 1=0$, we have to compute the limit of $l(x)= 2x\cosh (x)/\sinh(x) -2\log(\cosh (x))+ \log 2$. Here  
$$2x\frac{\cosh (x)}{\sinh (x)} -2\log(\cosh (x))= 
2\log\left( \frac{ \exp(x\cosh (x)/\sinh (x))}{\cosh x}\right). $$ 
Now remark that  $(x\cosh (x)-x\sinh (x))/\sinh (x)$  tends to zero, as 
$x\cosh( x)-x\sinh (x)= x\exp(-x)$. As $\exp(x)/\cosh x$  tends to $2$, by the above remarks we get that the the limit of  $l(x) $ is $ 2 \log 2 + \log 2= 3\log 2 >2$.
Therefore, the left side of inequality \eqref{lemineq2905} is proved. 
The right side follows by the fact that  $f(x)< f(x_1)$.  By Mathematica Software\textsuperscript{\tiny{\textregistered}} \cite{ru}, we can find 
$x_1 =1.606\ldots$ and $f(x_1)= 2.1312\ldots$.
\end{proof}

\noindent{\bf Proof of Theorem \ref{thm2}.} By Lemma \ref{BasicLemma}, one has  $(I/G)^2 = \exp(2(x/\tanh(x) -1))$, while $(A/G)^2= \cosh (x)^2,\, x>0.$
It is immediate that, the left side of \eqref{lemineq2905} implies the left side of 
\eqref{thm2ineq}.
Now, by the right side of \eqref{lemineq2905} one has  $I^2< \exp(c_1)(A^2+G^2)/2$, where  
$c_1= f(x_1)-2=0.13\cdots$. Since $\exp(0.13\cdots)= 1.14$, we get also the right side of 
\eqref{thm2ineq}. $\hfill\square$


\vspace{.4cm}
\noindent{\bf Proof of Theorem \ref{thm611}.} The first relation of \eqref{in0611} follows from the identity   
$$L(x^2,y^2)= ((x+y)/2)\cdot L(x,y),$$ 
which is a consequence of the definition of logarithmic mean, by letting $x=A, y=G$.   
The second inequality of \eqref{in0611} follows by \eqref{611ineqb}, while the third one is a consequence of the known inequality   
\begin{equation}\label{611ineqd}
L<(A+G)/2.
\end{equation}
A simple proof of \eqref{611ineqd} can be found in \cite{sandor611}. 
For \eqref{in0611a}, by the definition of logarithmic mean, one has  
$$L(I,G)=  (I-G)/\log (I/G),$$ 
and on base of the known identity 
$$\log (I/G) = A/L -1$$
(see \cite{sandorc,Seif1}),
we get    
$$L(I,G)=  ((I-G)/(A-L))L  < L,$$
since  the inequality   $(I-G)/(A-L)<1 $  can be rewritten as 
$$I+L  < A+ G $$
due to Alzer (see \cite{sandorc}).

The first inequality of \eqref{in-11} follows by the fact that  $L$ is a mean 
(i.e. if $x<y$ then   $x<L(x,y)<y$), and the well known relation $L< I$ 
(see \cite{sandorc})
For the proof of last relation of \eqref{in-11}  we will use a known inequality of 
S\'andor (\cite{sandorc}), namely:
\begin{equation}\label{in-12}
\log (I/L) >  1-  G/L.   
\end{equation}
Write now that $L(I,L)= (I-L)/\log (I/L)$, and apply \eqref{in-12}. Therefore, the proof of 
\eqref{in-11} is finished. $\hfill\square$

\vspace{.4cm}
\noindent{\bf Proof of Corollary \ref{coro}.}
The first inequality of \eqref{ineq711a} follows by the well known relation  
$L>\sqrt{GI}$ (see  \cite{alzer2}), while the second relation is a consequence of the classical relation $L(x,y)> G(x,y)$ (see e.g. \cite{sandorc}) applied to $x=I, y=G$. The last relation is inequality \eqref{in0611}.

The first inequality of \eqref{ineq711b} follows by \eqref{in0611}, while the second one by 
$L(I^2,G^2)= L(I,G)\cdot(I+G)/2$ and inequality $L< (I+G)/2 $. The last inequality follows in the same manner. $\hfill\square$

\vspace{.4cm}
\noindent{\bf Proof of Theorem \ref{thm-1011}.}
Since the mean $S$ is homogeneous, the relation  $I>S(A,G)$ may be rewritten as  
$I/G > S(A/G, 1)$, so by using logarithm and applying Lemma \ref{BasicLemma}, this inequality may be rewritten as 
\begin{equation}\label{3-in-1011}
\frac{x}{\tanh (x)}-1  >  \frac{\cosh (x) \log(\cosh(x))}{1+\cosh (x) },\quad  x>0.
\end{equation}
By using Mathematica Software\textsuperscript{\tiny{\textregistered}} \cite{ru}, one can see that inequality \eqref{3-in-1011} is not true for $x> 2.284$ . Similarly, the reverse inequality of \eqref{3-in-1011} is not true, e.g.
for $x<2.2$. 
These show that,  $I$ and $S(A,G)$ cannot be compared to each other.
In order to prove inequality \eqref{2-in-1011}, we will use the following result proved in 
\cite{rasa}:
The inequality
\begin{equation}\label{4-in-1011}
S>Q
\end{equation}
holds true. By writing \eqref{4-in-1011} as $S(a,b)>Q(a,b)$ for $a=Q,\, b= G$, and remarking that  $Q(a,b) = \sqrt{(a^2+b^2)/2}$  and that  $(Q^2+G^2)/2=  A^2$,  we get the first inequality of \eqref{2-in-1011}.
The second inequality is well known (see \cite{sandorc} for history and references).

By using $I(a,b)< A(a,b)= (a+b)/2$  for $a=Q$ and $b=G$ we get  
$I(Q,G) < (Q+G)/2$. On the other hand by inequality $(a+b)/2 < \sqrt{(a^2+b^2)/2}$  and  
$(Q^2+G^2)/2 = A^2$, inequality \eqref{3-in-1511} follows as well. This completes the proof.
$\hfill\square$


\end{document}